\documentclass[a4paper,11pt,leqno,english]{smfart}
\usepackage{aeguill}
\usepackage{enumerate}
\usepackage{amssymb,amsmath,latexsym,amsthm}
\usepackage[T1]{fontenc}
\usepackage{geometry}
\usepackage{url}
\usepackage[frenchb, main=english]{babel}
\usepackage[utf8]{inputenc}
\usepackage{mathrsfs}
\usepackage{xcolor}
\usepackage{comment}
\definecolor{violet}{rgb}{0.0,0.2,0.7}
\definecolor{rouge2}{rgb}{0.8,0.0,0.2}
\usepackage{tikz}
\usepackage{empheq}
\usepackage{tikz-cd}
\usetikzlibrary{matrix,arrows,decorations.pathmorphing}
\usepackage{hyperref}
\usepackage{mathpazo}
\hypersetup{
    bookmarks=true,         
    unicode=false,          
    pdftoolbar=true,        
    pdfmenubar=true,        
    pdffitwindow=false,     
    pdfstartview={FitH},    
    pdftitle={},    
    pdfauthor={},     
    colorlinks=true,       
   linkcolor=rouge2,          
    citecolor=violet,        
    filecolor=black,      
    urlcolor=cyan}           
\usepackage{enumitem}
\usepackage{appendix}

 \theoremstyle{plain}    
 \newtheorem{thm}{Theorem}[section]
\theoremstyle{plain} 
\newtheorem{bigthm}{Theorem}
\newtheorem{bigcoro}[bigthm]{Corollary}
 
 \numberwithin{equation}{section} 
 \numberwithin{figure}{section} 
 \newtheorem{cor}[thm]{Corollary} 
 \theoremstyle{plain}    
 \newtheorem{prop}[thm]{Proposition} 
 \theoremstyle{plain}    
 \newtheorem{lem}[thm]{Lemma} 
 \theoremstyle{remark}
 \theoremstyle{remark}
 \newtheorem{rem}[thm]{Remark}
 \theoremstyle{definition}

\theoremstyle{plain}  
\newtheorem{set}[thm]{Setting}
\theoremstyle{plain}

\theoremstyle{plain}

\theoremstyle{definition}
\newtheorem{defi}[thm]{Definition}

\newtheorem*{ackn}{Acknowledgements}

\newcommand{\B}{\mathbb{B}}
\newcommand{\C}{{\mathbb{C}}}
\newcommand{\N}{{\mathbb{N}}}

\def\1{\mathbf{1}}

\newcommand{\e}{\varepsilon}
\newcommand{\ep}{\varepsilon}

\newcommand{\om}{\omega}
\newcommand{\f}{\varphi}

\newcommand{\p}{\psi}

\newcommand{\omt}{\widetilde{\om}_{t}}

\newcommand{\Ric}{\mathrm{Ric}}

\newcommand{\reg}{\mathrm{\rm reg}}

\renewcommand{\ge}{\geqslant}
\renewcommand{\le}{\leqslant}

\newcommand{\sing}{\operatorname{\rm sing}}

\newcommand{\PSH}{\operatorname{PSH}}

\title{Continuity of singular K\"ahler-Einstein potentials}
\date{\today}

\author{Vincent Guedj}

\address{Institut de Mathématiques de Toulouse; UMR 5219, Université de Toulouse; CNRS, UPS, 118 route de Narbonne, F-31062 Toulouse Cedex 9, France \quad}

\email{vincent.guedj@math.univ-toulouse.fr}

\author{Henri Guenancia}
\address{Institut de Mathématiques de Toulouse; UMR 5219, Université de Toulouse; CNRS, UPS, 118 route de Narbonne, F-31062 Toulouse Cedex 9, France}
\email{henri.guenancia@math.cnrs.fr}

\author{Ahmed Zeriahi}

\address{Institut de Mathématiques de Toulouse; UMR 5219, Université de Toulouse; CNRS, UPS, 118 route de Narbonne, F-31062 Toulouse Cedex 9, France}

\email{ahmed.zeriahi@math.univ-toulouse.fr}

\begin{document}

\begin{abstract}  
In this note, we investigate some regularity aspects for solutions of degenerate complex Monge-Ampère equations (DCMAE) on singular spaces. 

First, we study the Dirichlet problem for DCMAE on singular Stein spaces,
 showing a general continuity result.
 A consequence of our results is that K\"ahler-Einstein potentials are continuous at isolated singularities.
 
 Next, we establish the global continuity of solutions to DCMAE
 when the reference class belongs to the real N\'eron-Severi group.
 This yields in particular 
 the continuity of K\"ahler-Einstein potentials on any irreducible Calabi-Yau variety.
\end{abstract} 

\maketitle

\tableofcontents

\section*{Introduction}

Generalizing Yau's solution of the Calabi conjecture \cite{Yau78},
singular K\"ahler-Einstein metrics 
on mildly singular K\"ahler varieties $X$
have been constructed in \cite{EGZ09} and further studied by many authors
(see \cite{GZbook, Bouck18,DonICM18} and the references therein).
These are honest K\"ahler forms $\omega_{\rm KE}=\omega+dd^c \f_{\rm KE}$
on the regular part $X_{\rm reg}$ of $X$, where $c_1(X)=\lambda [\omega]$ is proportional to a reference K\"ahler class
$[\omega]$, s.t.
$$
\Ric(\omega_{\rm KE})=\lambda \omega_{\rm KE}.
$$
One constructs $\omega_{\rm KE}=\omega+dd^c \f_{\rm KE}$ by solving a degenerate complex Monge-Amp\`ere equation
\begin{equation}
\label{eq KE}
(\omega+dd^c \f_{\rm KE})^n=e^{-\lambda \f_{\rm KE}} \mu_X,
\end{equation}
where $n=\dim_{\C} X$ and $\mu_X$ is an appropriate volume form.

Understanding the asymptotic behavior of $\f_{\rm KE}$ near $X_{\rm sing}$ is a major open problem.
  A very precise description 
  has been obtained by Hein-Sun in \cite{HS17},
 when the singularities are isolated and furthermore   isomorphic to a smoothable strongly regular Calabi-Yau cone.
 It would be highly desirable to extend this description to more general contexts, but this seems presently out of reach.

It has been shown in \cite{EGZ09} that $\f_{\rm KE}$ is always uniformly  bounded across $X_{\rm sing}$.
The viscosity approach developed in \cite{EGZ11} fails to establish continuity at the singularities
(see \cite{EGZ17}).
 In this article we use both local (Dirichlet problem analysis) and
 global (stability estimates) arguments to establish continuity properties 
 of $\f_{\rm KE}$, in various geometrical contexts that we now describe.

  \subsection*{Isolated singularities}
  Let $X$ be a Stein complex space which is reduced and locally irreducible, of complex dimension $n\ge 1 $, equipped  with a hermitian metric whose fundamental form is a positive $(1,1)$-form denoted by $\beta$; we do not assume that $\beta$ is closed here.
Let $\Omega \Subset X$ be a strongly pseudoconvex domain in $X$, i.e.  $\Omega$ admits a smooth
defining function $\rho$ which is stricly plurisubharmonic  in a neighborhood of  $\overline{\Omega}$.

Let $\phi$ be a continuous function on $\partial \Omega$ and $0 \le f \in L^p (\Omega,\beta^n)$ for some $p > 1$. 
The Dirichlet problem with boundary datum $\phi$ and right hand side $\mu := f \beta^n$ consists in finding a bounded plurisubharmonic function $U$ satisfying the following properties :   
\begin{equation}\label{eq:DirPb}
\left\{\begin{array}{lcl} 
 (dd^c U)^n =  \mu &\hbox{in }\  \Omega, \, \, \, \, \, 
 \\
  U_{\mid  \partial \Omega} = \phi, & \hbox{on}\  \partial \Omega, \, \, \, 
\end{array}\right.
\end{equation}
where the equation is here understood in the sense of Bedford-Taylor \cite{BT76,BT82,Bed82}, as we briefly recall in Section~\ref{sec:BT}.

\smallskip

When $X$, $\phi$ and $f$ are smooth and $f$ is positive, this problem has been addressed
by several authors (see notably \cite{CH98,GL10}) following
the fundamental work of Cafarelli-Kohn-Nirenberg-Spruck \cite{CKNS85}.
When $f$ merely belongs to $L^p$, $p>1$, a unique continuous solution has been constructed by Kolodziej \cite{Kol98}.
There is however not a very abundant literature dedicated to the case where $X$ is singular. Let us mention 
\cite{W09} which studies maximal plurisubharmonic functions, \cite{PS10} which treats the case of isolated singularities with RHS smoothly  degenerating to zero, \cite{EGZ09} which deals
with global Monge-Amp\`ere equations on mildly singular compact K\"ahler varieties and \cite{DLS} which seeks for solutions in Cegrell finite energy classes.

\smallskip

Our first main result is the following.

\begin{bigthm} \label{thmA}
The Dirichlet problem (\ref{eq:DirPb}) admits a unique solution $U \in \mathrm{PSH} (\Omega) \cap C^{0} (\overline{\Omega})$.  
\end{bigthm} 

The general strategy to prove Theorem~\ref{thmA} consists in producing the solution as the envelop of all subsolutions. This requires to approximate the data $(\phi, f)$ by smooth data and establish precise estimates along the regularization process. The main technical tool we use is the so-called stability estimate, cf Proposition~\ref{prop:stability}.

We apply this local theory to establish the continuity of singular K\"ahler-Einstein potentials
at isolated singularities as a consequence of Theorem~\ref{thmA}:

\begin{bigcoro} \label{corB}
A K\"ahler-Einstein potential $\f_{\rm KE}$, i.e. a solution of \eqref{eq KE}, is continuous at any isolated singularity of $X$.
\end{bigcoro} 

We show more generally that the results holds for K\"ahler-Einstein metrics on pairs, cf. Theorem~\ref{thm:isolatedKE}.

 \subsection*{Non-isolated singularities}
In order to deal with non-isolated singularities we require the reference
class $[\omega]$ to belong to the real N\'eron-Severi group, cf Section~\ref{sec NS} for the notation and Theorem~\ref{thm NS}.

  \begin{bigthm}  \label{thmC}
   Let $(X, \omega)$ be a compact normal Kähler space such that $[\omega] \in \widetilde{\mathrm{NS}}_{\mathbb R}(X) \subset H^1(X,\mathrm{PH}_X)$. Let $f\ge 0$ be a function in $L^p(X)$ for some $p>1$ such that $\int_X f\om^n=\int_X\om^n$. Let $\varphi\in \mathrm{PSH}(X,\om) \cap L^{\infty}(X)$ be the unique solution of the equation 
 $$(\om+dd^c\varphi)^n =f\omega^n$$
 such that $\sup_X \varphi =0$. Then $\varphi$ is continuous on $X$. 
 \end{bigthm}

 A few remarks are in order: 
 
 $-$ This result does not rely on the previous Theorem~\ref{thmA}, since we allow here arbitrary singularities. 
 
$- $ The continuity of $\f$ when $[\omega] \in \widetilde{\mathrm{NS}}_{\mathbb Q}(X)$ is a (rational) Hodge class is a consequence of the extension result of \cite{CGZ13}
 (some particular cases were obtained by Dinew-Zhang \cite{DZ10}). 
 
 $- $ If $X$ has rational singularities (e.g. if $X$ has klt singularities), the natural map $\beta:H^1(X,\mathrm{PH}_X)\to H^2(X,\mathbb R)$ is an injection and identifies  $ \widetilde{\mathrm{NS}}_{\mathbb R}(X)$ with the usual Néron-Severi group   $\mathrm{NS}_{\mathbb R}(X). $\\
 
 The proof of Theorem~\ref{thmC} goes as follows. We approximate $\om$ by a Kähler form $\om_{\ep}$ whose cohomology class is rational, and solve the equation $(\om_\ep+dd^c\varphi_\ep)^n= c_\ep f\om^n$ where $c_\ep$ is a harmless mass adjustment constant. By the remark above, $\varphi_\ep$ is continuous. Finally, one can show that $\varphi_\ep$ converges to $\varphi$ uniformly by a  suitable application of the maximum principle.  

 
 \smallskip
 

We finally apply Theorem \ref{thmC} to the case where $X$ is an irreducible Calabi-Yau variety $X$  (cf. Definition~\ref{def ICY} and Corollary~\ref{cor CY}). 

  \begin{bigcoro}  \label{thmD}
    Let $X$ be an irreducible Calabi-Yau variety,
    and let $\alpha$ be a Kähler class. 
    Then, the unique singular Ricci-flat metric $\omega_{\rm KE} \in \alpha$ has continuous potentials on $X$. 
  
 \end{bigcoro}


 \begin{ackn} 
The authors are grateful to the referee for carefully reading the manuscript and suggesting some improvements. This work has benefited from State aid managed by the ANR under the "PIA" program bearing the reference ANR-11-LABX-0040,
 in connection with the research project HERMETIC.
\end{ackn}

 \section{Preliminaries}
 
  \subsection{Plurisubharmonic functions on complex spaces}
  
  Let $X$ be a reduced complex analytic space  of pure dimension $n \ge 1$. 
 We will denote by $X_{\reg}$ the complex manifold of regular points of $X$. The set  
 $$
 X_{\sing} := X \setminus X_{\reg}
 $$
  of singular points is an analytic subset of $X$  of complex codimension $\ge 1$. 
 
 By definition for each point $x_0 \in X$ there exists a neighborhood $U$ of $x_0$ and a local embedding $j: U \hookrightarrow \mathbb C^N$ onto an analytic subset of  $\mathbb C^N$ for some $N \ge 1$.
 
 Using these local embeddings, it is possible to define the spaces of smooth forms of given degree on $X$ as smooth forms on $X_{\rm reg}$ that are locally on $X$ the restriction of an ambient form on $\mathbb C^N$.
  The notion of currents on $X$ is then defined by duality by their action on compactly supported smooth forms on $X$. The operators $\partial $ and $\bar{\partial }$, $d$, $d^c$ and $dd^c$ are then well defined by duality (see \cite{Dem85} for a careful treatment). 
  
  In the same way one can define the analytic notions of holomorphic and plurisubharmonic functions. 
  There are essentially two different notions :
 
 \begin{defi} 
 Let  $u : X \longrightarrow \mathbb R \cup \{-\infty\}$ be a given function.
 
 1. We say that   $u$ is  plurisubharmonic on $X$ if it is locally the restriction of a plurisubharmonic function on a local embedding of 
 $X$ onto an analytic subset of $\mathbb C^N$. 
 
 2. We say that  $u$ is  weakly plurisubharmonic on $X$ if $u$ is locally bounded from above on $X$ and  its restriction to the complex manifold $X_{\reg}$ is plurisubharmonic.
 \end{defi}

Fornaess and Narasimhan proved in \cite{FN} that   $u$ is plurisubarmonic on $X$ if and only if for any analytic disc 
$h : \mathbb D \longrightarrow X$, the restriction $u \circ h$ is subharmonic  on $\mathbb  D$ or identically $- \infty$.

If $u$ is  weakly plurisubharmonic on $X$,  $u$ is plurisubharmonic on $X_{\reg}$, hence upper semi-continuous on $X_{\reg}$. 
Since no assumption is made on $u$ at singular points, its natural to extend $u$  to $X$ by the following formula :
\begin{equation} \label{eq:extension}
u^* (x) := \limsup_{X_{\reg} \ni y \to x } \,  u (y), \, \, x \in X.
\end{equation}
 The function $u^*$  is upper semi-continuous, locally integrable on $X$ and  satisfies $dd^c u^* \ge 0$ in the sense of currents on $X$. 
By Demailly \cite{Dem85},  the two notions are equivalent when $X$ is locally irreducible. More precisely we will need the following result :

\begin{thm} \cite{Dem85}
Assume that $X$ is a locally irreducible analytic space and $u : X \longrightarrow \mathbb R \cup \{-\infty\}$
is a weakly plurisubharmonic  function on $X$, then  the function $u^*$ defined  by   (\ref{eq:extension})
is plurisubharmonic on $X$.
\end{thm}

Observe that since $u$ is plurisubharmonic on $X_{\reg}$, we have $u^* = u$ on $X_{\reg}$. 
Hence $u^*$ is the upper semi-continuous extension of $u|_{X_{\reg}}$  to $X$.
We note the following important consequence :

\begin{cor} \label{cor:uscregularization}
Let $\mathcal U \subset \mathrm{PSH} (X)$ be a non empty family of plurisubharmonic functions which is locally bounded from above on $X$. 
Then its upper  envelope
$$
U := \sup \{u \, ; \, u \in \mathcal U\},
$$
is a well-defined Borel function whose upper semi-continuous regularization $U^*$ is plurisubharmonic on $X$.
\end{cor}

This result is an extension of a classical result of P. Lelong concerning envelopes of plurisubharmonic functions on open sets in $\C^n$ 
(see \cite{Lel61,GZbook}).

\smallskip

 Following \cite{FN}  we say that $X$ is Stein if it admits a ${\mathcal C}^2$-smooth strongly plurisubharmonic exhaustion.
We will use the following definition :

\begin{defi} 
A domain $\Omega \Subset X$ is   strongly pseudoconvex if it admits a negative ${\mathcal C}^2$-smooth strongly plurisubharmonic exhaustion,
 i.e. a function $\rho$ strongly plurisubharmonic in a neighborhood $\Omega'$ of $\overline{\Omega}$ such that 
 $\Omega := \{ x \in \Omega' \, ; \, \rho (x) < 0\}$ and for any $c < 0$, 
 $$
 \Omega_c := \{x \in \Omega'; \, \rho (x) < c\} \Subset \Omega
 $$
 is relatively compact.
  \end{defi}
 
 Our complex spaces will be assumed to be reduced, locally irreducible of dimension $n \ge 1$. 
 We denote by $\mathrm{PSH}(X)$ the set of plurisubharmonic functions on $X$. Finally, we fix a smooth, strictly positive $(1,1)$-form $\beta$ on $X$.

 \subsection{Monge-Amp\`ere operator on singular spaces}  \label{sec:BT}
 
  The complex Monge-Amp\`ere measure $(dd^c u)^n$ of a smooth psh function in a domain of $\C^n$  is 
  the   Radon measure
  $$
  (dd^c u)^n=c \det \left( \frac{\partial^2 u}{\partial z_i \partial \overline{z}_j} \right) dV_{\rm eucl},
  $$
  where $c>0$ is a normalizing constant. The definition has been extended to any bounded psh function
  by Bedford-Taylor, who laid down the foundations of {\it pluripotential theory} in \cite{BT76,BT82}
  (see \cite{GZbook} for a recent treatment).

 The complex Monge-Amp\`ere operator has been defined and studied 
 on complex spaces
 by Bedford in \cite{Bed82} and later by Demailly in \cite{Dem85}. 
  It turns out that if $u \in \mathrm{PSH}(X) \cap L^{\infty}_{\rm loc} (X)$, 
  the  Monge-Amp\`ere measure $(dd^c u)^n$ is well defined on $X_{\rm reg}$
   and can be extended  to $X$ as a Borel  measure with zero mass on $X_{\rm sing}$.
 
 Thus all standard properties of the complex Monge-Amp\`ere operator  acting on $ \mathrm{PSH}(X) \cap L^{\infty}_{\rm loc} (X)$ 
 extend to this setting (see \cite{Bed82,Dem85}). In particular :
 
 
 \begin{prop} \label{prop:comparisonpple} 
 Let  $\Omega \Subset X$ a relatively compact open set and $u, v \in \mathrm{PSH}(\Omega) \cap L^{\infty} (\Omega). $ 
 Assume that $\liminf_{x \to \zeta} (u(x) - v(x)) \ge 0$ for any $\zeta \in \partial \Omega$.
  Then
 $$
 \int_{\{u < v\}} (dd^c v)^n \le \int_{\{u < v\}} (dd^c u)^n.
 $$  
 In particular, if $(dd^c u)^n \le (dd^c v)^n$ weakly on $\Omega$, then $v \le u$ on $\Omega$.
 \end{prop}
 
 The proof relies on Stokes formula  (see \cite{Bed82}). 
  As a first step towards proving Theorem~\ref{thmA}, we treat the following simpler
 Dirichlet problem with {\it bounded density} and {\it zero boundary value}. This will be a key result to prove the subextension property (Lemma~\ref{lem:Subext}), on which relies, in turn, the central stability estimate (Proposition~\ref{prop:stability}). 
 
 \begin{lem} \label{lem:DP0} 
 Let $\Omega \Subset X$ be a   strongly pseudoconvex domain
and $0 \leq f \in L^{\infty} (\Omega)$ a bounded density. 
 There exists a unique function $w \in \PSH(\Omega) \cap L^{\infty}(\Omega)$ such that 
 $$
 (dd^c w)^n = f \beta^n
 \; \; \text{ on } \; \; \Omega,
 \; \; \text{ with } \; \;
 w_{|\partial \Omega}=0.
 $$
 \end{lem}

 \begin{proof}
  Up to scaling $\rho$, one can assume that $dd^c \rho \ge \beta$. The function $v_A  := A \rho $ is plurisubharmonic and smooth
   in a neighborhood of  $\overline{\Omega}$, with ${v_A}_{|\partial \Omega}=0$.
 Moreover    $(dd^c v_A)^n \geq  A^n \beta^n \ge f \beta^n$ in $\Omega$
    if $A \geq \Vert f\Vert_{\infty}^{1 \slash n}$, hence
    $v_A$ is  a subsolution to the Dirichlet problem if $A$ is large enough.
    
   We consider the envelope of subsolutions  
   $
   w := \sup \{v ; v \in \mathcal S \},
   $
   where
   $$
  \mathcal S :=  \{v  \, ; \,  v \in \PSH (\Omega) \cap L^{\infty} (\Omega),  \, \, (dd^c v)^n \geq f \beta^n \text{ and } v_{|\partial \Omega} \leq 0 \}.
   $$
   It follows from the maximum principle that $w \leq 0$, while $w \geq v_A$ since  $v_A \in \mathcal S$.
   
   A classical argument guarantees that    $w$ is   plurisubharmonic 
 and again a subsolution to the Dirichlet problem i.e.  $(dd^c w)^n \geq f \beta^n$ on $\Omega$ and $w = 0$ in $\partial \Omega$
 (see the proof of Lemma \ref{lem:MaxSubSol} where this is explained in  a slightly more general context).
   
   It remains to show that  $(dd^c w)^n = f \beta^n$ in $\Omega_{\rm reg}$. This follows from a  balayage process. 
Choose a chart $(B,z)$ in a neighborhood of given regular point $a  \in \Omega_{\rm reg}$ such that there exists an isomorphism 
$F : B \rightarrow \B $  from a neighborhood of $\bar B$ onto a neighborhood of the closed euclidean unit ball $\overline \B$ with $F (B) = \B$.  
We consider the Dirichlet problem 
\begin{equation} \label{DirPb0}
(dd^c u)^n = f \beta^n, \, \, \hbox{on} \, \, \, B \, \, \, \hbox{and} \, \, \, u_{\mid \partial B} = w.  
\end{equation}
Observe that $w$ is a subsolution. 
Using  a result of Cegrell
\cite{Ce84}, we can find a solution $u_B \in \mathrm{PSH} (B) \cap L^{\infty} (B) $ to this problem (see also \cite[Theorem 5.17]{GZbook}).  
It follows from the comparison principle (Proposition \ref{prop:comparisonpple}) that $ w \le u_B$ on $B$.

The function $u$ defined on $\Omega$ by $ u = u_B$ on $B$ and $u = w $ on $\Omega \setminus B$ 
is a plurisubharmonic  subsolution to the Dirichlet problem (\ref{DirPb0}) , since the smooth volume form  $\beta^n$ puts no mass on $\partial B$. Therefore  $ u \le w$ on $\Omega$.
This yields    $u = w = u_B$ on $B$, hence $(dd^c w) ^n = f \beta^n$ in $B$. 
Since $B\subset \Omega_{\rm reg}$ was arbitrary, it follows that $(dd^c w)^n  = f \beta^n$ on $\Omega_{\rm reg}$, 
hence on the whole $\Omega$ since none of the two measures puts any mass on $\Omega_{\rm sing}$.
 \end{proof}

  \subsection{Stability estimates}  \label{sec:stability}
  
  \subsubsection{Subextension}
 
 We now explain a useful subextension property of plurisubharmonic functions with finite Monge-Amp\`ere mass.   In the sequel we set for a bounded open subset $\Omega\subset X$
  $$
  \mathcal E^0 (\Omega)=
 \left\{\phi \in \PSH (\Omega) \cap L^{\infty} (\Omega), \; 
  \phi_{|\partial \Omega} = 0
  \text{ and } \int_\Omega (dd^c \phi)^n < + \infty \right\}.
  $$

 \noindent
  Let  now $\Omega \Subset \widetilde{\Omega} \Subset X$ be two strongly pseudoconvex domains.

 \begin{lem}  \label{lem:Subext}
 Given $u \in \mathcal E^0 (\Omega)$ we extend it to $\widetilde \Omega$ by $0$ outside $\Omega$ and consider  
 $$
 \tilde u := \sup \{ v \in \PSH (\widetilde \Omega) \ ; \  \, v \leq {\bf 1}_{\Omega} \, u \, \, \, \hbox{in} \, \, \, \widetilde \Omega\}.
 $$
Then  $\tilde u \in \mathcal  E^0 (\widetilde \Omega)$, $\tilde u \leq u$ in $\Omega$ and 
  $\int_{\widetilde \Omega} (dd^c \tilde u)^n \leq \int_{\Omega} (dd^c u)^n .$
 \end{lem}
 
 This statement is proved in \cite{CZ03} for domains in $\C^n$. 
 The proof follows the same idea with a twist relying on \cite{FN} and Lemma~\ref{lem:DP0} as we explain below.
 
  \begin{proof} 
Let $\tilde \rho$ be a plurisubharmonic defining function of $\widetilde \Omega$. Since $u $ is bounded in $\Omega \Subset \widetilde \Omega$, 
one can find $A \gg 1$ such that $A \tilde \rho \leq u \leq 0$ in $\Omega$. 
 Thus  $\tilde u$ is a well defined bounded plurisubharmonic function on $\widetilde \Omega$ such that $A \tilde \rho  \leq \tilde u \leq 0$ in $\widetilde \Omega$ and $\tilde u \leq u$ in $\Omega$. 
 
 We need to estimate the Monge-Ampère mass of $\tilde u$.
  The main idea is to construct a subsolution with at most the same Monge-Ampère mass by solving the complex Monge-Ampère equation $(dd^c v)^n = {\bf 1}_\Omega (dd^c u)^n$ on $\widetilde \Omega$.  
  We are not yet able to solve such an equation unless the right hand side is a measure with bounded density thanks to Lemma \ref{lem:DP0},
  so we proceed by approximation.
  
Assume first that $u$ is plurisubharmonic  in a Stein neighborhood $\Omega'$ of $\bar \Omega$.  
A result of Fornaess-Narasimhan \cite{FN} ensures the existence of
 a sequence $(u_j)$ of smooth plurisubharmonic functions on $\Omega'$ which decreases to $u$.
As the measure $\mu_j :=  (dd^c  u_j)^n = f_j \beta^n$ has bounded density  $f_j \in L^{\infty} ( \Omega)$,
it follows from Lemma \ref{lem:DP0} that one can  solve the Dirichlet problem $(dd^c v_j)^n = {\bf 1}_{\Omega} (dd^c  u_j)^n$ on $\widetilde \Omega$ with boundary value $ v_j = 0$ in $\partial{ \widetilde \Omega}$. Since $ v_j \leq 0 \leq  u_j$ in $\partial \Omega$, 
it follows from Proposition \ref{prop:comparisonpple} on $\Omega$, that $ v_j \leq u_j$ in $\Omega$.
 Moreover $\int_{\widetilde \Omega} (dd^c v_j)^n = \int_{\Omega} (dd^c  u_j)^n$. 
  By definition  $v_j \leq \tilde u_j \leq 0$ in $\widetilde \Omega$ and   $v_j = \tilde u_j = 0$ in $ \partial \widetilde \Omega$. 
Proposition \ref{prop:comparisonpple} thus yields $\int_{\widetilde \Omega} (dd^c \tilde u_j)^n \leq  \int_{\widetilde \Omega} (dd^c v_j)^n = \int_{\Omega} (dd^c  u_j)^n$.
 
 Since   $(u_j)_j$ decreases to $u$ in a neighborhood of $\bar \Omega$ and $A \rho \leq u$ in $\Omega$,
  the sequence $(\tilde u_j)$  decreases to a bounded plurisubharmonic function   $v $ such that 
  $A \tilde \rho \leq v \leq 0$ in $\widetilde \Omega$ and $v \leq u$ in $\Omega$. 
  Moreover  $\int_{\widetilde \Omega} (dd^c v)^n \leq \int_{\bar{\Omega}} (dd^c  u)^n < \infty$ hence $v \in \mathcal E^0 (\widetilde \Omega)$.  
  Observe that $ v \leq \tilde u$ in $\widetilde \Omega$ and $v = \tilde u = 0$ in $\partial \widetilde \Omega$. Hence $\int_{\widetilde \Omega} (dd^c \tilde u)^n \leq \int_{\Omega} (dd^c v)^n \leq  \int_{\bar{\Omega}} (dd^c  u)^n < \infty$.
 
To treat the general case we let $(\Omega_j)_j$ be an exhaustive sequence of bounded strongly pseudoconvex domains in $\Omega$
and set
 $$
 w_j :=  \sup \{ v \in \PSH (\widetilde \Omega) \ ; \ v < 0 \, \, \hbox{and} \, \, v \leq u \, \hbox{on} \, \, \Omega_j\}.
 $$
 Since $ A \tilde \rho \leq u$ in $\Omega$, it follows that $w_j$ is a bounded plurisubharmonic function in $\widetilde \Omega$ such that $A \tilde \rho \leq \tilde u \leq  w_j \leq {\bf 1}_{\Omega_j} u$ in $\widetilde \Omega$. 
 Moreover $(w_j)$   decreases   to a bounded psh functions $w $ such that  $ A \tilde \rho \leq \tilde u \leq w \leq {\bf 1}_{\Omega} u$ in $\widetilde \Omega$. Hence  $w \leq \tilde u$ in $\widetilde \Omega$.
  This proves that $w = \tilde u$ in $\widetilde \Omega$.
  
 The previous case ensures that
 $
  \int_{\widetilde \Omega} (dd^c w_j)^n   \leq   \int_{\bar \Omega_j} (dd^c u)^n \leq \int_{\Omega} (dd^c u)^n,
$
hence
 $$ \int_{\widetilde \Omega} (dd^c \tilde u)^n \leq \liminf_{j \to + \infty}   \int_{\widetilde \Omega} (dd^c w_j)^n \leq  \int_{\Omega} (dd^c u)^n.
 $$
This shows that $\tilde u \in \mathcal E^0(\widetilde \Omega)$ and $ \int_{\widetilde \Omega} (dd^c \tilde u)^n \leq \int_{\Omega} (dd^c u)^n.$
\end{proof}

  \subsubsection{Stability}
 
 We now state an important stability estimate that we shall use repeatedly in the sequel:

\begin{prop} \label{prop:stability}
Let  $\Omega \Subset X$ a relatively compact open set  and $u, v \in \mathrm{PSH}(\Omega) \cap L^{\infty} (\Omega). $ 
Let $\beta$ be a smooth positive $(1,1)$-form on $X$ and let $f\in L^p(\Omega)$ for some $p>1$. 
 Assume that 
 \[(dd^c u)^n = f\beta^n \quad \mbox{on } \,\, \Omega.\]
Then for any $0 <\gamma < 1 \slash (nq + 1)$, there exists a constant $C = C(\gamma, \|f\|_p) > 0$ such that
$$
\sup_{\Omega}(v-u)_+ \le \sup_{\partial \Omega}(v-u)^*_+  \, + \, C \Vert (v-u)_+\Vert_1^\gamma.
$$
where $\Vert (v-u)_+\Vert_1 := \int_\Omega (v-u)_+ \beta^n$, and $q$ is the conjugate exponent of $p$, i.e. $\frac 1p+\frac 1q=1$.
\end{prop}

Here  $w_+ := \sup \{w ,0\}$ and $w^*$ means the upper semi-continuous extension of a bounded function $w$ on $\Omega$ 
to the boundary i.e. $w(\zeta) := \limsup_{z \to \zeta} w (z)$.

 This result is proved in  \cite[Theorem 1.1]{GKZ08} when the ambient space is $X=\C^n$
 (see also \cite[Proposition 3.3]{EGZ09} for a similar result in the context of compact K\"ahler varieties).
 The proof relies on capacity estimates and the use of the comparison principle (Proposition \ref{prop:comparisonpple}).
 It can be reproduced identically in our present context, once the following 
 volume-capacity comparison has been established:
 
  \begin{lem}  \label{lem:VolCapIneq} 
 For any $u \in \mathcal \PSH \cap L^{\infty}(\Omega)$
 such that $\int_{\Omega} (dd^c u)^n \leq 1$ and $u_{|\partial \Omega}=0$,
\begin{equation} \label{eq:expint}
\int_\Omega e^{-  u} \beta^n \leq C=C (n,\Omega).
\end{equation}
In particular, we have the following inequality for any Borel  set $K \subset \Omega$
\begin{equation} 
\label{eq:VolCapIneq}
{\rm Vol}_\beta (K) \leq C \exp \left(-  \, {\rm Cap} (K,\Omega)^{-1\slash n}\right).
\end{equation}
\end{lem}
 
Here $\hbox{Vol}_\beta (K) := \int_K \beta^n$ and $\hbox{Cap} (K,\Omega)$ denotes the Monge-Amp\`ere capacity.

 \begin{proof}[Sketch of proof of Lemma~\ref{lem:VolCapIneq}]
 The fact that \eqref{eq:VolCapIneq} follows from \eqref{eq:expint} is standard by choosing $u$ to be a multiple of the extremal function $h_{K,\Omega}$ of $K$, which satisfies $\int_\Omega (dd^c h_{K,\Omega}^*)^n=\mathrm{Cap}(K,\Omega)$, cf~\cite[Theorems 4.32 \& 4.34]{GZbook}.
 
 The proof of \eqref{eq:expint} goes roughly as follows. Choose a strongly pseudoconvex domain $\widetilde \Omega$ (with strictly psh exhaustion function $\tilde \rho$) containing $\bar \Omega$ in its interior, and consider the subextension $\tilde u $ of $u$ constructed in Lemma~\ref{lem:Subext}. Recall that $\tilde u\in \mathcal E^0(\widetilde \Omega)$ and $\int_{\widetilde \Omega} (dd^c \tilde u)^n \le 1$. Since $\tilde u \le u$ on $\Omega$, one has $\int_\Omega e^{-u}\beta^n \le \int_{\bar \Omega}e^{-\tilde u} \beta^n$ hence it is sufficient to bound that last quantity. 
 
 For any element $v\in \mathcal E^0(\widetilde \Omega)$, one has $\int_{\widetilde \Omega} (- v)^n (dd^c \tilde\rho)^n\le c_n \|\tilde \rho\|^n_{\infty}\int_{\widetilde \Omega}(dd^c  v)^n$ for some universal constant $c_n$. In particular, 
 the set $\Sigma:=\{v\in \mathcal E^0(\widetilde \Omega); \int_{\widetilde \Omega} (dd^c v)^n \le 1\}$ is relatively compact in $\PSH(\widetilde \Omega)$ for the $L^1_{\rm loc}$ topology and one can check that for any $v\in \overline \Sigma$, $v$ has a well-defined Monge-Ampère and  $ \int_\Omega (dd^c v)^n \le 1$, cf \cite[Corollary~4.4]{Zer01}. 
 
 Now, take a resolution $p:\widehat \Omega\to \widetilde \Omega$, choose a strictly positive $(1,1)$-form $\hat \beta \ge p^*\beta$ on $\widehat \Omega$  and fix a compact set $K\subset \widetilde \Omega$. By the property above, functions $\hat v\in p^*\overline\Sigma$ have Monge-Ampère mass at most one, hence $\sup_{x\in p^{-1}(K)} \nu(\hat v,x) \le 1$. By covering $p^{-1}(K)$ with finitely many coordinate charts, one can use \cite[Corollary~3.2]{Zer01} to get a constant $C_K>0$ such that $\int_{p^{-1}(K)}e^{-\hat v}\hat\beta^n\le C_K$ for any $\hat v\in p^*\overline \Sigma$. The conclusion now follows by choosing $K=\bar \Omega$ and $\hat v=p^*\tilde u$. 
 \end{proof}

 \section{Dirichlet problem on singular complex spaces} \label{sec:Dircont}
 
 In this Section, we will work in the following
 \begin{set}
 \label{set Dir}
 Let $X$ be a Stein space of dimension $n\ge 1$, reduced and locally irreducible. We fix a smooth positive $(1,1)$-form $\beta$ on $X$, not necessarily closed. Let $\Omega \Subset X$ be a pseudoconvex domain with $\mathcal C^2$-smooth exhaustion function $\rho$, let $f\in L^p(\Omega)$ for some $p>1$ and let $\mu:=f\beta^n$. Finally, let $\phi \in \mathcal C^0(\partial \Omega)$. 
 \end{set}
 
 By the comparison principle Proposition \ref{prop:comparisonpple}, if the Dirichlet problem (\ref{eq:DirPb}) admits a solution,
 it dominates any subsolution. We will now prove that the maximal subsolution is psh, hence coincides with the solution in that case.

\begin{defi}
A plurisubharmonic function $v \in \mathrm{PSH}^{\infty} (\Omega) := \mathrm{PSH} (\Omega) \cap L^{\infty} (\Omega)$ is a subsolution to  the Dirichlet problem (\ref{eq:DirPb}) with data $(\phi,\mu)$ if the following two conditions are satisfied :
\begin{enumerate}
\item $v^* (\zeta) := \limsup_{\Omega \ni z \to \zeta} v (z) \le \phi (\zeta)$, for any $\zeta \in \partial \Omega$,
\item $(dd^c v)^n \ge \mu$ weakly in $\Omega$.
\end{enumerate}
\end{defi}

It is   natural to consider the family $\mathcal{S}_{\phi,\mu}  = \mathcal{S}_{\phi,\mu} (\Omega)$ of all subsolutions to the Dirichlet problem (\ref{eq:DirPb}) with data $(\phi,\mu)$ and its upper envelope when it exists.

\subsection{The subsolution property}

We start with the following standard observation.

\begin{lem}  \label{lem:MaxSubSol}
In Setting~\ref{set Dir}, assume that the Dirichlet problem (\ref{eq:DirPb}) with data $(\phi,\mu)$ admits a subsolution $v_0 \in \mathcal{S}_{\phi,\mu} (\Omega)$. 
Then the  upper envelope of subsolutions  
$$
U := U_{\phi,\mu} := \sup \{v \ ; \ v \in \mathcal{S}_{\phi,\mu} (\Omega)\}
$$
is a subsolution to the Dirichlet problem (\ref{eq:DirPb}), i.e. $ U_{\phi,\mu}  \in \mathcal{S}_{\phi,\mu} (\Omega)$. 

\noindent
Moreover, if $v_0 = \phi$ on $\partial \Omega$ then
\begin{equation}\label{eq:SubSol}
  \lim_{z \to \zeta} U (z) = \phi (\zeta),  \quad \mbox{for any } \,\,  \zeta \in \partial \Omega.
\end{equation}
\end{lem}

The proof of the lemma will actually not make use of the assumption $f\in L^p(\Omega)$, but we chose to stick to Setting~\ref{set Dir} throughout this section for clarity. 

\begin{proof}
Observe that for any $u \in \mathcal{S}_{\phi,\mu} (\Omega)$ we have $ u \le M_\phi:= \max_{\partial \Omega} \phi$ on $\Omega$. 
Therefore, the envelope $U := U_{\phi,\mu}$ is a well-defined bounded function on $\Omega$ and $U^*$ is a bounded plurisubharmonic function on $\Omega$ by Corollary  \ref{cor:uscregularization}.
It follows then   from  a topological lemma of Choquet  and the comparison principle Proposition \ref{prop:comparisonpple},   
 that $U^*$ satisfies the inequality $(dd^c U^*)^n \ge \mu$ in $\Omega$ (see \cite[Theorem~5.12, Step 1]{GZbook}). It remains to show that $U=U^*\in \mathrm{PSH}(\Omega)$ and that the boundary values of $U$ are $\le \phi$ (resp. $=\phi$ in the second case). \\

\noindent
{\it Claim.}   Let $U_\phi := U_{\phi,{\bf 0}}$ be the maximal plurisubharmonic function on $\Omega$  with boundary values $\le \phi$. Then, $U_\phi=U_\phi^* \in \mathrm{PSH}(\Omega)$ and it satisfies
\begin{equation} \label{eq:MaxSubSolphi}
\lim_{z \to \zeta} U_\phi (z) = \phi (\zeta), \quad \mbox{for any} \,\, \zeta \in \partial \Omega.
\end{equation}

For the time being, assume that the claim holds and let us finish the proof of Lemma~\ref{lem:MaxSubSol}.  Since we have 
\begin{equation}
\label{boundary}
v_0\le U\le U_\phi \quad \mbox{on } \,\,  \Omega
\end{equation}
the claim would imply $v_0\le U^*\le U_\phi$ on $\Omega$. Thus, one would get  $U^* \in \mathcal{S}_{\phi,\mu} (\Omega)$,
hence $U^* \le U$ on $\Omega$, and finally $U^* = U$ on $\Omega$. The boundary condition follows from \eqref{boundary}.\\

It now remains to prove the claim. To establish \eqref{eq:MaxSubSolphi}, we fix $\varepsilon > 0$ and let $\psi$ be a $\mathcal C^2$ function on $\partial \Omega$ such that $\psi - \varepsilon \le \phi \le \psi$ on $\partial \Omega$.
By definition  
$$
U_\psi - \varepsilon \le U_\phi \le  U_\phi^* \le U_\psi^* \quad \text{on} \, \, \Omega,
$$
reducing the problem to the case when $\phi$ is $\mathcal C^2$ on $\partial \Omega$.
 In this case we can  extend $\phi$ as a $ \mathcal C^2$ function in a neighborhood of $\overline{\Omega}$. 
Choosing $A > 1$ large enough, we can assume that
 $v_{\phi} := A \rho + \phi$ is plurisubharmonic and $\mathcal C^2$ in a neighborhood of  $\overline{\Omega}$ and 
 $v_{\phi} = \phi$ on $\partial \Omega$. Hence $v_\phi \le U_\phi$. 
Similarly we choose  $B > 1$ large enough so that   $ w_\phi := - B \rho + \phi$ is a  $\mathcal C^2$ plurisuperharmonic function 
in a neighborhood of $\overline{\Omega}$ such that $w_{\phi} = \phi$ on $\partial \Omega$.
It follows from the maximum principle  that $v_\phi \le U_\phi \le U_\phi^*  \le w_\phi$ on $\Omega$.
 This implies that $U^*_\phi$ is a subsolution hence $U_\phi^*\le U_\phi$ and finally $U_\phi=U_\phi^* \in \mathrm{PSH}(\Omega)$. At the same time, we get that $U_\phi$ satisfies \eqref{eq:MaxSubSolphi}. The claim and lemma follow.
\end{proof}

\subsection{Existence of a solution}

We now show that $U_{\phi,\mu}$ is a solution of the Dirichlet problem (\ref{eq:DirPb})
using  Lemma \ref{lem:MaxSubSol} and a balayage process in $\Omega_{\reg}$,
adapting the proofs of \cite[Theorem 8.3]{BT76} and \cite[Theorem 5.16]{GZbook}.   

\begin{thm}  \label{thm:contsoldir}
In Setting~\ref{set Dir}, we have $\mathcal{S}_{\phi,\mu} (\Omega) \neq \emptyset$ and its upper envelope 
\begin{equation} \label{eq:envelope}
U:=  \sup  \{v  \, ; \, v \in \mathcal{S}_{\phi,\mu} (\Omega)\} 
\end{equation}
satisfies $U\in  \mathrm{PSH}(\Omega)\cap \mathcal C(\Omega)$. Moreover, $U$ admits a continuous extension to $\overline \Omega$ and it solves the Dirichlet problem \eqref{eq:DirPb}.
\end{thm}

When $f \equiv 0$, a similar result has been obtained by Wikstr\"om \cite[Theorem 1.8]{W09}.

\begin{proof} 
We proceed in several steps.

\smallskip

\noindent {\it  Step 1: Assume first
that $\phi \in \mathcal C^2 (\partial \Omega)$ and $\mu = f \beta^n$ with  $f \in L^{\infty} (\Omega)$.}
We claim that there exists a subsolution  $v$  plurisubharmonic and smooth in a neighborhood of  $\overline{\Omega}$ such that $v = \phi$ on $\partial \Omega$.
  Indeed  consider any smooth extension of  $\phi$ to   a neighborhood of $\overline \Omega$. 
  For $A > 1$ large enough, the function $v_\phi = v_\phi^A := A \rho + \phi$ is plurisubharmonic and smooth
   in a neighborhood of  $\overline{\Omega}$ and satisfies   $(dd^c v_\phi)^n \ge f \beta^n$ in $\Omega$,
   since $f$ is bounded from above.
   Since  $v_\phi = \phi$ on $\partial \Omega$
we can apply  Lemma \ref{lem:MaxSubSol}  
to conclude that  $(dd^c U)^n \ge \mu$ on $\Omega$ and  
$$
\lim_{z \to \zeta} U (z) = \phi (\zeta) \quad \text{for any} \, \, \zeta \in \partial \Omega.
$$

The identity  $(dd^c U)^n = \mu$ in $\Omega $ can be obtained by a balayage argument exactly as in the proof of Lemma~\ref{lem:DP0}. 
%

\smallskip

To show the continuity of $U$ on $\overline{\Omega}$ we first observe that $U$ can be extended as a plurisubharmonic function
 on a neighborhood $\Omega'$ of $\overline{\Omega}$. Indeed since $v_{\phi}$ is plurisubharmonic on a neighborhood $\Omega'$ of $\overline{\Omega}$ 
 and $v_{\phi} \le U$ on $\Omega$ with $v_\phi = \phi = U$ on $\partial \Omega$,
  the function $V$ defined by $V = U$ on $\Omega$ and $V := v_\phi$ on $\Omega' \setminus \Omega$ is plurisubharmonic on $\Omega'$. We can always assume that $\Omega'$ is  a Stein space. By Fornaess-Narasimhan approximation Theorem \cite[Theorem 5.5]{FN} there exists a decreasing sequence $(V_j)$ of smooth plurisubharmonic functions on $\Omega'$ which converges to $V$ pointwise on $\Omega'$.
Consider now   $v_j := V_j$ on $\overline{\Omega}$. Then $v_j \in \mathrm{PSH} (\Omega) \cap C^0 (\overline{\Omega})$ and  $(v_j)$ decreases to $U$ on $\overline{\Omega}$. In particular, $v_j \to U$ in $L^1$ by the monotone convergence theorem.  It follows from Proposition \ref{prop:stability} that
\begin{equation} \label{eq:tsability2)}
\sup_{\Omega} (v_j - U) \le \sup_{\partial \Omega} (v_j - U) + C \Vert v_j - U\Vert_1^\gamma,
\end{equation}
where $0 < \gamma < 1 \slash (n q + 1)$. 
Since $(v_j)$ decreases to $U = \phi$ on $\partial \Omega$,
 Dini's lemma ensures that the convergence is uniform on $\partial \Omega$. 
 It follows therefore from (\ref{eq:tsability2)}) 
 that $(v_j)$ uniformly converges  to $U$ on $\overline{\Omega}$, hence $U$ is continuous on $\overline{ \Omega}$.

\medskip

\noindent {\it Step 2 : Assume $\phi \in C^0 (\partial \Omega)$ and $\mu = f \beta^n$ with  $f \in L^{\infty} (\Omega)$.}
Take a decreasing sequence of smooth functions $(\phi_j)_{j \in \N^*}$ converging to $\phi$ on $\partial \Omega$ so that $\phi_j -  1 \slash j \le \phi \le \phi_j$ on $\partial \Omega$ for any $j \ge 1$.    Then from the definitions we see that
$$
 U_{\phi_j,\mu} - 1 \slash j \le  U_{\phi,\mu} \le  U_{\phi_j,\mu}, \,  \, \, \text{on} \, \, \, \overline{\Omega}.
$$
  This implies that the sequence $(U_{\phi_j,\mu})$ uniformly converges  to $ U_{\phi,\mu}$ on $\overline \Omega$.  Hence $ U = U_{\phi,\mu}$ is continuous  on $\overline \Omega$, $U= \phi$ on $\partial \Omega$  and satisfies the equation
 $ (dd^c U)^n = \mu$ on $\Omega$ thanks to \cite[Corollary~3.6]{DemaillyBook}.

\medskip
 
\noindent {\it Step 3 :  Assume $\phi \in C^0 (\partial \Omega)$ and $\mu = f \beta^n$ with  $f \in L^{p} (\Omega)$, with $p>1$}.
 Define $f_j := \min \{f,j\}$ for $j \in \N$. The bounded densities $(f_j)$ increase   to $f$ in $L^p (\Omega)$.
  Set $U_j :=U_{\phi,\mu_j}$ where $\mu_j := f_j \beta^n$ for $j\in \N$.
By the comparison principle Proposition \ref{prop:comparisonpple},   $(U_j)_{j \in \N}$ is a decreasing sequence of bounded plurisubharmonic functions on $\Omega$ that are continuous up to the boundary. 

We first show that the $U_j$'s are uniformly bounded.
Let $v := U_{\phi,{\bf 0}}$ denote the maximal plurisubharmonic function on $\Omega$ with boundary value $\phi$
(whose properties have been established the proof of Lemma~\ref{lem:MaxSubSol}). 
The comparison principle Proposition \ref{prop:comparisonpple} ensures that
 $U_j \le v$ on $\Omega$.  It follows from  Proposition \ref{prop:stability} that for all $j \in \N$,
$$
\Vert v - U_j \Vert_{L^{\infty}} \le C \Vert v - U_j\Vert_{L^1}^\gamma,
$$ 
where $0 < \gamma < 1$ and  $C > 0$ is a uniform constant depending only on a uniform bound of $\Vert f_j\Vert_p \le \Vert f\Vert_p$.
This implies in particular that    $\Vert v - U_j \Vert_{L^{\infty}}^{1 - \gamma} \le C'$, 
where $C' > 0$ is a uniform constant, hence $\sup_j \Vert U_j \Vert_{L^{\infty}} < + \infty$.

The sequence   $(U_j)_{j \in \N}$ of continuous functions on $\overline \Omega$ therefore decreases to a bounded function $V$ such that $V|_\Omega$ is plurisubharmonic and $V|_{\partial \Omega} = \phi$ since $U_j|_{\partial \Omega} = \phi$ for all $j$. 
 Moreover $(dd^c V)^n = \mu$, as follows from the continuity property of the complex Monge-Amp\`ere operator
along monotone sequences, cf \cite[Theorem~3.18]{GZbook}.
Thus $V$ is a subsolution to (\ref{eq:DirPb}).
By Lemma \ref{lem:MaxSubSol}, the function $U = U_{\mu,\phi}$ is a subsolution  to (\ref{eq:DirPb}).
 Since  $U$ is the maximal subsolution we have $V \le  U$.
On the other hand $(dd^c U)^n \ge (dd^c V)^n$ so
Proposition \ref{prop:comparisonpple} ensures $U \le V$ on $\Omega$. 
Therefore $U = V$ and $(dd^c U)^n=\mu$. In particular, $U_j$ converges to $U$ in $L^1$ by the monotone convergence theorem. 

The continuity of $U$  again follows from the stability estimate (Proposition \ref{prop:stability}) : 
there exists a constant $C > 0$ such that for all $j \in \N$
$$
\sup_{\Omega} (U_j - U) \le  C \Vert U_j - U\Vert_1^\gamma,
$$
showing that $(U_j)$  uniformly  converges to $U$ on $\overline{\Omega}$, hence $U$ is continuous on $\overline{\Omega}$ 
and solves the Dirichlet problem (\ref{eq:DirPb}).
\end{proof}

 \section{Continuity of K\"ahler-Einstein potentials} \label{sec:KE}

 \subsection{Generalities}

 Let $X$ be a normal compact space. We say that $X$ is Kähler if it admits a Kähler form in the following sense 
 
  \begin{defi}
 \label{defi kahler}
  A Kähler form $\omega$ on a normal complex space $X$ can be defined as a Kähler form on $X_{\rm reg}$ which extends to a smooth, closed $(1,1)$-form under local embeddings $X\underset{\rm loc}{\hookrightarrow} \mathbb C^N$. 
 In particular, one can cover $X$ with open subsets $U_\alpha$, find embeddings $j_\alpha:U_\alpha \hookrightarrow \mathbb C^N$ as well as smooth, strictly psh functions $\varphi_{\alpha}$ defined on a neighborhood of $j_{\alpha}(U_\alpha)$ such that $\omega|_{U_\alpha}=j_\alpha^*dd^c\varphi_\alpha$. 
 \end{defi}
 
 The study of complex Monge-Amp\`ere equations in this context has been initiated
 in \cite{EGZ09}, providing a way of
 constructing singular K\"ahler-Einstein metrics and extending Yau's fundamental solution to the Calabi conjecture \cite{Yau78}. More precisely, it is proven there that given a Kähler metric $\omega$ on $X$ and a non-negative function $f\in L^p(X)$ for some $p>1$ satisfying $\int_X f\om^n = \int \om^n$, then the equation
 \begin{equation}
 \label{MA sing}
 (\omega+dd^c \varphi^n)=f\cdot \omega^n
 \end{equation}
 has a unique solution $\varphi \in \mathrm{PSH}(X,\omega) \cap L^{\infty}(X)$ such that $\sup_X \varphi =0$.\\
 
 Let us now explain the relation between the equation \eqref{MA sing} above and the existence of singular Kähler-Einstein metrics.
 
  We choose a pair $(X,D)$ consisting of an $n$-dimensional compact K\"ahler variety $X$ and a divisor $D=\sum a_i D_i$ with $a_i \in [0,1] \cap \mathbb Q$. We assume that there exists an integer $m\ge 1$ such that $m(K_X+D)$ is a line bundle.

 Given a hermitian metric $h$ on $K_X+D$ and the singular metric $e^{-\phi_D}$ (only unique up to a positive multiple), one construct a measure $\mu_{(X,D),h}$ on $X$ as follows. If $U$ is any open set where $m(K_X+D)$ admits a trivialization $\sigma$ on $U_{\rm reg}$, then the expression $$\frac{i^{n^2}(\sigma \wedge \bar \sigma)^{\frac 1m}}{|\sigma|_{h^{\otimes m}}^{\frac 2m}}e^{-\phi_D}$$
defines a measure on $U_{\rm reg}$ which is independent of $m$ as well as the choice of $\sigma$ and can thus be patched to a measure on $X_{\rm reg}$. Its extension by $0$ on $X_{\rm sing}$ is by definition $\mu_{(X,D),h}$. We recall the following properties satisfied by the measure $\mu:=\mu_{(X,D),h}$, cf \cite[Lemma~6.4]{EGZ09}.
\begin{itemize}
\item The Ricci curvature of $\mu$ on $X_{\rm reg}$ is equal to $-i\Theta(h)+[D]$.
\item The mass $\int_X d\mu$ is finite if and only if $(X,D)$ has klt singularities. 
\item If $\mu$ has finite mass, then the density $f$ of $\mu$ wrt $\om^n$ (i.e. $\mu=f\cdot \om^n$) satisfies $f\in L^p(X)$ for some $p>1$. 
\end{itemize}

\medskip
From now on, we work in the following 

\begin{set}
\label{set pairs}
Let $(X,D)$ be a pair where $X$ is a compact normal Kähler space and $D$ is an effective $\mathbb Q$-divisor. Assume that $(X,D)$ has klt singularities, pick a Kähler metric $\omega$ and a hermitian metric $h$ on $K_X+D$, normalized so that $\int_X d\mu_{(X,D),h}=\int_X\om^n$. We assume either
\begin{itemize}
\item[$\bullet$]  $K_X+D$ is ample  and $\omega = i\Theta(h)$; \, or
\item[$\bullet$] $K_X+D \equiv 0$ and $h$ satisfies $i\Theta(h)=0$; \, or 
\item[$\bullet$]  $K_X+D$ is anti-ample  and $\omega = -i\Theta(h)$.
\end{itemize}
\end{set}

\begin{defi}
\label{def KE}
 In the Setting~\ref{set pairs} above, a Kähler-Einstein metric is a solution $\om_{\rm KE}:=\omega+dd^c \varphi_{\rm KE}$ of the Monge-Ampère equation
\begin{equation}
\label{KE sing}
 (\omega+dd^c \f_{\rm KE})^n =e^{-\lambda \f_{\rm KE}} \mu_{(X,D),h}
 \end{equation}
where $\lambda = -1,0,1$ according to whether we are in the first, second of third case.
 It  satisfies
 \begin{equation}
 \label{KE sing 2}
 {\rm Ric}(\omega_{\rm KE})=\lambda \omega_{\rm KE}+[D]
 \end{equation}  in the weak sense.
 \end{defi}
 
By the results \cite{EGZ09} recalled above, \eqref{KE sing} admits a unique solution $\om_{\rm KE}$ whenever $\lambda \in \{-1,0\}$. Its potential $\varphi_{\rm KE}$ is globally bounded on $X$ and $\om_{\rm KE}$ is a honest Kähler-Einstein metric on $X_{\rm reg} \setminus \mathrm{Supp}(D)$ and it has cone singularities along $D$ generically \cite{G2}. 
For $\lambda=+1$ we refer the reader to \cite{BBEGZ,Bouck18}.\\

 Understanding the asymptotic behavior of $\f_{\rm KE}$ (resp. $\omega_{\rm KE}$) near $X_{\rm sing}$ is a major open problem.
 A  breakthrough has been obtained in \cite{HS17} for special type of Calabi-Yau varieties,
 but  a lack of local models prevent us from a good understanding in more general situations.
 We discuss below two different tools that allow one to show -in many contexts- global continuity of the 
 K\"ahler-Einstein potential $\f_{\rm KE}$, a statement that was overlooked in \cite{EGZ11}, cf \cite{EGZ17}.

 \subsection{Isolated singularities}
 
 
 \begin{thm} \label{thm:isolatedKE}
Let $(X,D)$ be a pair as in Setting~\ref{set pairs}. Then the potential $\f_{\rm KE}$ of any Kähler-Einstein metric in the sense of Definition~\ref{def KE}  is continuous near an isolated singularity of $X$.
 \end{thm}
 
 \begin{proof}
We work near  an isolated singular point $a$.
 We let $B$ denote a small strictly pseudoconvex neighborhood of 
 $a$ in $X$, isomorphic to the trace of a ball in some local embedding in $\C^N$,
 and let $\rho$ denote a local smooth potential for $\omega=dd^c \rho$ in $B$.
 
 Recall from \cite{EGZ17} that the K\"ahler-Einstein potential 
 $\f_{\rm KE}$ is continuous in $B \setminus \{a\}$.
We can apply the local theory developed in Section~\ref{sec:Dircont}
to obtain that $\psi=\rho+\f_{\rm KE}$ is the unique solution of the Dirichlet problem
$$
(dd^c \p)^n=e^{-\lambda \f_{\rm KE}} \mu_{(X,D),h} \,\, \text{ in } B,
\; \; \text{ with } \; \; 
\p_{|\partial B}=(\rho+\f_{\rm KE})_{|\partial B}.
$$
Recall that $\mu_{(X,D),h}$ has an $L^p$ density wrt $\omega^n$ since $(X,D)$ is klt. It follows therefore from Theorem \ref{thm:contsoldir} that $\psi$ 
 is continuous, hence so is $\f_{\rm KE}$ at point $a$.
 \end{proof}

  \subsection{Hodge classes vs transcendental classes}
  \label{sec NS}
  Let $X$ be a compact Kähler space. 
  In $H^2(X,\mathbb R)$, an important subgroup is made up by the first Chern class of line bundles $L$ on $X$ via the map 
  \[ c_1:H^1(X,\mathcal O_X^*) \to H^2(X,\mathbb R)\]
   induced by the exponential exact sequence
   \begin{equation}
   \label{exp ES}
   0\longrightarrow \underline{\mathbb Z} \longrightarrow \mathcal O_X \overset{e^{2\pi i\cdot}}{\longrightarrow}  \mathcal O_X^{*}\longrightarrow 0
   \end{equation}
   and the sheaf injection $\underline{\mathbb Z} \hookrightarrow \underline{\mathbb R}$. 
   
   \begin{defi}
   The group $\mathrm{Im}(c_1) \subset  H^2(X,\mathbb R)$ is called Néron-Severi group of $X$ and denoted by $\mathrm{NS}(X)$. One defines the vector spaces
   $\mathrm{NS}_{\mathbb Q}(X):=\mathrm{NS}(X) \otimes \mathbb Q \subset H^2(X,\mathbb R)$ and  $\mathrm{NS}_{\mathbb R}(X):=\mathrm{NS}(X) \otimes \mathbb R \subset H^2(X,\mathbb R)$. 
   \end{defi}

 A Kähler metric $\omega$ (cf Definition~\ref{defi kahler}) is canonically attached to an element in $H^0(X,\mathcal C^{\infty}_X/\mathrm{PH}_X)$ where $\mathcal C^{\infty}_X$ (resp. $\mathrm{PH}_X$) is the subsheaf of continous functions on $X$ that are local restrictions of smooth functions (resp. pluriharmonic functions) under local embeddings $X\underset{\rm loc}{\hookrightarrow} \mathbb C^N$. It can be proved that a pluriharmonic function is locally the real part of a holomorphic function, \cite[\S 4.6.1]{BEG}. The exact sequences 
 
 \[   0\longrightarrow \mathrm{PH}_X \longrightarrow \mathcal C^{\infty}_X  \longrightarrow \mathcal C^{\infty}_X/\mathrm{PH}_X \longrightarrow 0\]
 
 \[   0\longrightarrow \underline{ \mathbb R} \longrightarrow \mathcal O_X \overset{\mathrm{Im}(\cdot)}{\longrightarrow} \mathrm{PH}_X \longrightarrow  0\]
 yield a map
 \[H^0(X,\mathcal C^{\infty}_X/\mathrm{PH}_X) \overset{[\cdot]}{\longrightarrow} H^1(X,\mathrm{PH}_X)\overset{\beta}{\longrightarrow} H^2(X,\mathbb R)\]

\begin{defi}
A class $\alpha \in H^1(X, \mathrm {PH}_X)$ is called Kähler if there exists a Kähler metric $\omega$ on $X$ such that $\alpha=[\omega]$. The Kähler cone of $X$ is the set $\mathcal K_X\subset H^1(X,\mathrm{PH}_X)$ made out of Kähler classes. It is an open, convex set. 
\end{defi}

\begin{rem}
\label{ES}
When $X$ has rational singularities, it is proved in \cite[Remark~3.2 (2)]{GK20} that $\beta$ is an injection and that we have an exact sequence
\[0\longrightarrow H^1(X,\mathrm{PH}_X) \overset{\beta}{\longrightarrow} H^2(X,\mathbb R) \longrightarrow  H^2(X, \mathcal O_X) \longrightarrow 0.\] 
In particular, under those assumptions, one can alternatively view the Kähler cone $\mathcal K_X \subset H^2(X,\mathbb R)$ without any ambiguity.
\end{rem}

Finally, the logarithm map yields an exact sequence
\[  0\longrightarrow \underline{ \mathbb R/\mathbb Z} \overset{e^{2\pi i \cdot}}{\longrightarrow} \mathcal O_X^* \overset{\log |\cdot|}{\longrightarrow} \mathrm{PH}_X \longrightarrow  0\]
hence a map
$$H^1(X,\mathcal O_X^*) \overset{\gamma}{\longrightarrow} H^1(X,\mathrm{PH}_X)$$
such that $\beta \circ \gamma = c_1$. 

\begin{defi}
 One set  $\widetilde{\mathrm{NS}}(X):=\mathrm{Im}(\gamma) \subset  H^1(X,\mathrm{PH}_X)$. One also defines the vector spaces
   $\widetilde{\mathrm{NS}}_{\mathbb Q}(X):=\widetilde{\mathrm{NS}}(X) \otimes \mathbb Q \subset H^1(X,\mathrm{PH}_X)$ and  $\widetilde{\mathrm{NS}}_{\mathbb R}(X):=\widetilde{\mathrm{NS}}(X) \otimes \mathbb R \subset H^1(X,\mathrm{PH}_X)$.
\end{defi}

The relation between the two groups introduced is that 
\begin{equation}
\label{NS}
\mathrm{NS}(X)=\beta(\widetilde{\mathrm{NS}}(X))
\end{equation}
 and when $X$ has rational singularities, then Remark~\ref{ES} implies that $\beta$ induces an isomorphism between the two groups and, in particular, between their tensorization by $\mathbb Q$ or $\mathbb R$. 
 
 A generalization of Kodaira theorem shows that a line bundle $L\to X$ on a normal compact space $X$ such that $\gamma(L)\in \mathcal K_X$ is automatically ample, cf e.g. \cite[Proposition~5.12]{EGZ09}. In particular, there exists an embedding $f:X\hookrightarrow \mathbb P^N$ such that $f^*\mathcal O_{\mathbb P^N}(1)=L$ and $[f^*\omega_{\rm FS}]=\gamma(L)$. 
 
 One can slightly extend the result above by asserting that if $\widetilde{\mathrm{NS}}_{\mathbb R}(X) \cap \mathcal K_X \neq \emptyset$, then $X$ is projective. Indeed,  $\mathcal K_X$ is open and $\widetilde{\mathrm{NS}}_{\mathbb Q}(X)$ is dense in $\widetilde{\mathrm{NS}}_{\mathbb R}(X)$, so $\widetilde{\mathrm{NS}}_{\mathbb Q}(X) \cap \mathcal K_X \neq \emptyset$. Since $\mathcal K_X$ is a cone, we further deduce $\widetilde{\mathrm{NS}}(X) \cap \mathcal K_X \neq \emptyset$ and we can apply the result above. 
 
  
  \medskip 
  
 Coming back to Equation~\eqref{MA sing}, it is proved in \cite{CGZ13}  that if $[\omega] \in  \widetilde{\mathrm{NS}}_{\mathbb Q}(X)$, then the potential solution $\f$ can be approximated by a decreasing 
 sequence of smooth $\omega$-psh functions, hence \cite[Theorem 2.1]{EGZ09} ensures that $\f$ is globally continuous on $X$. 

In particular, the previous discussion applies to show that if $(X,D)$ is a klt pair such that $K_X+D$ has a sign, then a solution $\f_{\rm KE}$ of \eqref{KE sing} is continuous on $X$ as soon as $\lambda \neq 0$. That is, singular Kähler-Einstein metrics $\omega_{\rm KE}$ with positive or negative Ricci curvature have continuous potentials. \\
 
 The remaining case of interest is when the curvature is zero (i.e. $\lambda =0$) and $[\omega]$ is an arbitrary Kähler class in $H^1(X,\mathrm{PH}_X)$. The following results deals with the case $[\omega]\in \widetilde{\mathrm{NS}}_{\mathbb R}(X)$. 
 
 \begin{thm}
 \label{thm NS}
 Let $(X, \omega)$ be a compact normal Kähler space such that $[\omega] \in \widetilde{\mathrm{NS}}_{\mathbb R}(X) \subset H^1(X,\mathrm{PH}_X)$. Let $f\ge 0$ be a function in $L^p(X)$ for some $p>1$ such that $\int_X f\om^n=\int_X\om^n$. Let $\varphi\in \mathrm{PSH}(X,\om) \cap L^{\infty}(X)$ be the unique solution of the equation 
 $$(\om+dd^c\varphi)^n =f\omega^n$$
 such that $\sup_X \varphi =0$. Then $\varphi$ is continuous on $X$. 
 \end{thm}
 
 Generalizing the celebrated Beauville-Bogomolov decomposition, it was recently proved in \cite{BGL} that compact Kähler spaces $X$ with klt singularities such that $K_X$ is numerically trivial admit a finite cover $X'\to X$, unramified in codimension one and such that $X'=T\times \prod Y_i \times \prod Z_j$ where $T$ is a (smooth) torus, $Y_i$ are irreducible Calabi-Yau varieties and $Z_j$ are irreducible holomorphic symplectic varieties, cf. also \cite{Dru16,GGK, HP,CGGN} for related anterior results. We will focus on the irreducible Calabi-Yau factors, whose definition we now recall.
 
 \begin{defi}
 \label{def ICY}
Let $X$ be a compact, normal Kähler space of dimension $n\ge 3$. We say that $X$ is  an irreducible Calabi-Yau variety if $X$ has canonical singularities, trivial canonical bundle and is such that for any finite quasi-étale cover $X'\to X$, we have $h^0(X',\Omega_{X'}^{[p]})=0$ for any $1\le p \le n -1$. 
\end{defi}
Since there are no reflexive two-forms on an irreducible Calabi-Yau variety $X$, it is automatically projective. From Theorem~\ref{thm NS}, we deduce the following
 \begin{cor}
 \label{cor CY}
 Let $X$ be an irreducible Calabi-Yau variety and let $\alpha \in \mathcal K_X$ be a Kähler class. Then, the unique singular Ricci-flat metric $\omega \in \alpha$ has continuous potentials on $X$. 
 \end{cor}
 
 \begin{proof}[Proof of Corollary~\ref{cor CY}]
Let $p:\widetilde X \to X$ be a resolution of singularities. Since $X$ has rational singularities and $\widetilde X$ is compact Kähler, we have for any $p\le n$
\begin{align*}
h^p(X,\mathcal O_X)=h^p(\widetilde X, \mathcal O_{\widetilde X}) 
=h^0(\widetilde X,\Omega_{\widetilde X}^{p}).
\end{align*}
By \cite{KS}, we have $p_*\Omega_{\widetilde X}^p=\Omega_X^{[p]}$, and therefore
\[h^p(X,\mathcal O_X)= h^0(X,\Omega_{X}^{[p]})=0\] for $p=1,2$. The exact sequence \eqref{exp ES} shows that the integral first Chern class yields an isomorphism $H^1(X,\mathcal O_X^*) \simeq H^2(X,\mathbb Z)$ so, in particular, one has
\begin{equation}
\label{H2}
 \mathrm{NS}_{\mathbb R}(X)=H^2(X,\mathbb R).
 \end{equation}
 By \eqref{NS} and the remark below it, if follows that $\widetilde{\mathrm{NS}}_{\mathbb R}(X)=H^1(X,\mathrm{PH}_X)$. The corollary is now proved.
\end{proof}
 
 \begin{proof}[Proof of Theorem~\ref{thm NS}]
 It follows from \cite[Theorem 4.1]{EGZ09} that there is a
  unique solution $u\in \mathrm{PSH}(X,\om) \cap L^{\infty}(X)$ of the equation
 \begin{equation}
 \label{MA exp}
 (\om+dd^c u)^n =e^u\cdot f\omega^n
 \end{equation}
We claim it is continuous. If that claim holds, we can apply it to $\varphi$ which is the solution of the equation
 $$ (\om+dd^c u)^n =e^u\cdot (fe^{-\varphi})\omega^n$$
 This is indeed legitimate since $fe^{-\varphi} \in L^p(X)$ as $\varphi$ is bounded. 
 
 We are now going to prove the claim. Since $\widetilde{\mathrm{NS}}_{\mathbb Q}(X)$ is dense is $\widetilde{\mathrm{NS}}_{\mathbb R}(X)$, one can find a family of Kähler metrics $\omega_\ep$ such that
 \begin{itemize}
 \item[$\bullet$] $ [\om_\ep] \in \widetilde{\mathrm{NS}}_{\mathbb Q}(X)$
 \item[$\bullet$] $[\om_\ep] \to [\om]$ in $H^1(X,\mathrm{PH}_X)$. 
  \item[$\bullet$] $\om_\ep \to \om$ smoothly on $X$. 
 \end{itemize}
 
 Indeed, $H^1(X,\mathrm{PH}_X)$ is a finite-dimensional vector space and it is the image of the surjective map $[\cdot]: H^0(X, \mathcal C^{\infty}_X/\mathrm{PH}_X)\to H^1(X,\mathrm{PH}_X)$. Clearly, $[\omega]$ is a non-zero element, hence one can find smooth $(1,1)$-forms $\om_1, \ldots, \om_r$ with local potentials such that $([\omega], [\omega_1], \ldots, [\omega_r])$ is a basis of $H^1(X,\mathrm{PH}_X)$. In order to uniformize the notation, we set $\om_0:=\om$. 
 
 By assumption, $[\omega]$ is a limit of rational classes. This means that one can find a $(r+1)$-tuple $\lambda_\ep=(\lambda_{i,\ep})_{0\le i \le r}$ of real numbers converging to $(1,0, \ldots, 0)$ and such that $\sum_{i=0}^r \lambda_{i,\ep} [\omega_i]$ is a rational class. If one sets $\om_\ep:=\sum_{i=0}^r \lambda_{i,\ep} \om_i$, then $\om_{\ep}$ is a Kähler metric for $\ep$ small enough and it clearly satisfies the desired conditions stated above. In particular, there is no loss of generality in assuming that we have the following set of inequalities
 \begin{equation}
 \label{control}
 (1-\ep) \om \le \om_\ep \le (1+\ep) \om
 \end{equation}
 We consider the unique function $u_\ep \in \mathrm{PSH}(X,\om) \cap L^{\infty}(X)$ solution of the equation
 \begin{equation}
 \label{MA exp ep}
 (\om_\ep+dd^c u_\ep)^n =e^{u_\ep}\cdot f\omega^n
 \end{equation}
As we explained above, $u_\ep \in \mathcal C^0(X)$ since $ [\om_\ep] \in \widetilde{\mathrm{NS}}_{\mathbb Q}(X)$. Moreover, by Jensen's inequality, if follows easily that 
\begin{equation}
\label{sup}
\sup_X u_{\ep} \le C
\end{equation}
 for some $C>0$ independent of $\ep$. Moreover, if follows from \eqref{control} that the function $v_\ep:=(1-\ep)u+n\log(1-\ep)-\ep \|u\|_{\infty}$ is $\om_\ep$-psh and satisfies
 \begin{align*}
 (\om_\ep+dd^c v_\ep)^n & \ge (1-\ep)^n (\om+dd^c u)^n \\
 & =e^{(1-\ep)u+n\log (1-\ep)+\ep u} \cdot f \om^n \\
 & \ge e^{v_\ep} \cdot f\om^n.
 \end{align*}
 This shows that $v_\ep$ is a subsolution of \eqref{MA exp ep} and we get 
 \begin{equation}
 \label{eq below}
 v_\ep = (1-\ep)u+n\log(1-\ep)-\ep \|u\|_{\infty} \le u_\ep
 \end{equation}
 In particular, we have a uniform lower bound for $u_\ep$ independent of $\ep$. Combined with \eqref{sup} we obtain, up to adjusting $C$, the following control
\begin{equation}
  \label{Linfty}
\| u_{\ep}\|_{\infty} \le C
\end{equation}
Finally, we consider $w_\ep:=\frac 1{1+\ep} u_\ep-n\log(1+\ep) - \frac{\ep}{1+\ep}\|u_\ep\|_{\infty}$. By \eqref{control}, the function $w_\ep$ is $\om$-psh and satisfies
 \begin{align*}
 (\om+dd^c w_\ep)^n & \ge (1+\ep)^{-n} (\om_\ep+dd^c u_\ep)^n \\
 & = e^{\frac{1}{1+\ep}u_\ep-n\log(1+\ep) +\frac{\ep}{1+\ep}u_\ep}\cdot f\om^n \\
 & \ge e^{w_\ep} \cdot f\om^n.
 \end{align*}
  This shows that $w_\ep$ is a subsolution of \eqref{MA exp} and we get 
 \begin{equation}
 \label{eq above}
 w_\ep = \frac 1{1+\ep} u_\ep-n\log(1+\ep) - \frac{\ep}{1+\ep}\|u_\ep\|_{\infty}\le u.
 \end{equation}
Combining \eqref{eq below}-\eqref{eq above} and \eqref{Linfty}, we see that
\begin{align*}
\|u-u_\ep\|_{\infty}  \le -n \log(1-\ep)+2\ep \max(\|u\|_{\infty}, \|u_\ep\|_{\infty}) 
=O(\ep).
\end{align*}
That is, $u$ is the uniform limit of the continuous functions $u_\ep$ when $\ep\to 0$. In particular, it is continuous. 
 \end{proof}

 \bibliographystyle{smfalpha}
\bibliography{biblio_CKE}

\end{document}